\documentclass[12pt,a4paper,oneside]{amsart}
\usepackage[a4paper]{geometry}
\usepackage{mathptmx}
\DeclareMathAlphabet{\mathcal}{OMS}{cmsy}{m}{n} 

\usepackage{enumerate,mathrsfs,perpage}

\DeclareMathOperator{\Der}{Der}
\DeclareMathOperator{\Homol}{H}
\DeclareMathOperator{\id}{id}
\DeclareMathOperator{\im}{Im}
\DeclareMathOperator{\Ker}{Ker}

\newcommand{\vertbar}{\>|\>}
\newcommand{\set}[2]{\ensuremath{\{ #1 \vertbar #2 \}}}

\MakePerPage[2]{footnote}

\newtheorem*{theorem}{Main Theorem}
\newtheorem{lemma}{Lemma}
\newtheorem*{conjecture}{Conjecture}

\theoremstyle{definition}
\newtheorem*{remark}{Remark}

\begin{document}

\title{A $\delta$-first Whitehead Lemma}

\author{Arezoo Zohrabi}
\address{University of Ostrava, Czech Republic}
\email{azohrabi230@gmail.com}

\author{Pasha Zusmanovich}
\email{pasha.zusmanovich@osu.cz}

\date{First written June 19, 2020; last minor revision November 12, 2022}
\thanks{J. Algebra \textbf{573} (2021), 476--491}

\begin{abstract}
We prove that $\delta$-derivations of a simple finite-dimensional Lie algebra
over a field of characteristic zero, with values in a finite-dimensional module,
are either inner derivations, or, in the case of adjoint module, multiplications
by a scalar, or some exceptional cases related to $\mathfrak{sl}(2)$. This can 
be viewed as an extension of the classical first Whitehead Lemma.
\end{abstract}

\maketitle

\section{Introduction}

Let $L$ be a Lie algebra, $V$ an $L$-module, with the module action denoted by
$\bullet$, and $\delta$ is an element of the base field. Recall that a 
\emph{$\delta$-derivation of $L$ with values in $V$} is a linear map 
$D: L \to V$ such that
\begin{equation}\label{eq-eqdelta}
D([x,y]) = - \delta y \bullet D(x) + \delta x \bullet D(y)
\end{equation}
for any $x,y \in L$. The set of all such maps for a fixed $\delta$ forms a 
vector space which will be denoted by $\Der_\delta(L,V)$.

In the case $V = L$, the adjoint module, we speak about just 
\emph{$\delta$-derivation of $L$}. The latter notion generalizes simultaneously
the notions of derivation (ordinary derivations are just $1$-derivations) and of
centroid (any element of the centroid is, obviously, a $\frac 12$-derivation).

$\delta$-derivations of Lie and other classes of algebras were a subject of an
intensive study (see, for example, \cite{filippov-5}--\cite{filippov}, 
\cite{hopkins}, \cite{delta}, \cite{leger-luks} and references therein; the 
latter paper is devoted to a more general notion of so-called quasiderivations 
which we do not discuss here). As a rule, algebras from ``nice'' classes 
(simple, prime, Kac--Moody, Lie algebras of vector fields, etc.) possess a very
``few'' nontrivial $\delta$-derivations. On the other hand, there are very few results, 
if at all, about $\delta$-derivations with values in modules.

The aim of this paper is to prove that $\delta$-derivations of simple 
finite-dimensional Lie algebras of characteristic zero, with values in
finite-dimensional modules, are, as a rule, just inner $1$-derivations. The 
exceptional cases are identity maps with values in adjoint modules (which are 
$\frac 12$-derivations), or are related to $\mathfrak{sl}(2)$ which, unlike all 
other simple Lie algebras, possesses nontrivial $\delta$-derivations. (Note the
occurrence of the new exceptional values of $\delta$, in addition to 
the exceptional values $\delta=-1,\frac 12, 1$ previously known from the 
literature). The exact statement runs as follows.

\begin{theorem}
Let $\mathfrak g$ be a semisimple finite-dimensional Lie algebra over an 
algebraically closed field $K$ of characteristic $0$, $V$ a finite-dimensional 
$\mathfrak g$-module, and $\delta \in K$. Then $\Der_\delta(\mathfrak g, V)$ is
nonzero if and only if one of the following holds:
\begin{enumerate}[\upshape(i)]
\item 
$\delta = 1$, in which case $\Der_1(\mathfrak g, V) \simeq V$ and consists of 
inner derivations of the form  $x \mapsto x \bullet v$ for some $v \in V$.

\item 
$\delta = -\frac 2n$ for some integer $n \ge 1$, or $\delta = \frac{2}{n+2}$ for
some integer $n \ge 3$, or $\delta = \frac 12$, and $V$ is decomposable into the
direct sum of irreducible $\mathfrak g$-modules in such a way that each direct 
summand of $V$ is a nontrivial irreducible module over exactly one of the simple direct 
summands of $\mathfrak g$, and a trivial module over the rest of them. 
\end{enumerate}

In the latter case, decomposing $\mathfrak g$ into the direct sum of simple 
algebras: $\mathfrak g = \mathfrak g_1 \oplus \dots \oplus \mathfrak g_m$, and 
writing 
$$
V = V_1^{1} \oplus \dots \oplus V_1^{k_1} \oplus
    V_2^{1} \oplus \dots \oplus V_2^{k_2} \oplus \dots \oplus
    V_m^{1} \oplus \dots \oplus V_m^{k_m} ,
$$
where $V_i^j$ is an irreducible module over $\mathfrak g_i$, and a trivial 
module over $\mathfrak g_\ell$, $\ell \ne i$, we have:
$$
\Der_{-\frac 2n}(\mathfrak g, V) \simeq 
\bigoplus\limits_{\mathfrak g_i \simeq \mathfrak{sl}(2)} \>
\bigoplus\limits_{V_i^j \simeq V(n)} \mathscr D_i^j ,
$$
where $n \ge 1$, $V(n)$ is the $(n+1)$-dimensional irreducible 
$\mathfrak{sl}(2)$-module, and $\mathscr D_i^j$ is the 
$(n+3)$-di\-men\-si\-o\-nal vector space of $(-\frac 2n)$-derivations of $\mathfrak g_i \simeq \mathfrak{sl}(2)$ with values
in $V_i^j \simeq V(n)$, as described in Lemma \ref{lemma-2}(\ref{it-2n}).

$$
\Der_{\frac{2}{n+2}}(\mathfrak g, V) \simeq 
\bigoplus\limits_{\mathfrak g_i \simeq \mathfrak{sl}(2)} \>
\bigoplus\limits_{V_i^j \simeq V(n)} \mathscr E_i^j,
$$
where $n \ge 3$, $\mathscr E_i^j$ is the $(n-1)$-dimensional vector space of 
$\frac{2}{n+2}$-derivations of $\mathfrak g_i \simeq \mathfrak{sl}(2)$ with 
values in $V_i^j \simeq V(n)$, as described in 
Lemma~\ref{lemma-2}(\ref{it-2n2}).

$$
\Der_{\frac 12}(\mathfrak g, V) \simeq 
\bigoplus\limits_{i=1}^m \> \bigoplus\limits_{V_i^j \simeq \mathfrak g_i} 
K_i^j ,
$$
where the inner summation is carried over all occurrences of $V_i^j$ being 
isomorphic to the adjoint module $\mathfrak g_i$, and $K_i^j$ is the 
one-dimensional vector space spanned by this isomorphism, considered as a map
$\mathfrak g_i \to V_i^j$.
\end{theorem}

Since $\delta$-derivations do not change under field extensions, this theorem 
essentially describes $\delta$-de\-ri\-va\-t\-i\-ons of a semisimple 
finite-dimensional Lie algebra with values in a finite-dimensional module, over 
an arbitrary field of characteristic zero. However, the formulation in the case
of an arbitrary field would involve forms of algebras and modules in the 
$\mathfrak{sl}(2)$-related cases, and elements of centroid instead of identity 
maps in the case $\delta = \frac 12$, and would be even more cumbersome, so we 
confine ourselves with the present formulation.

The classical first Whitehead Lemma states that for $\mathfrak g$ and $V$ as
in the statement of the main theorem, the first cohomology vanishes: 
$\Homol^1(\mathfrak g, V) = 0$. As the first cohomology is interpreted as the
quotient of derivations of $\mathfrak g$ with values in $V$ modulo inner 
derivations, and ordinary derivations are just $1$-derivations, this theorem can
be viewed as an extension of the first Whitehead Lemma. The standard proof of 
the first Whitehead Lemma involves the Casimir operator (see, for example, 
\cite[Chapter III, \S 7, Lemma 3]{jacobson}) and will not work in the case 
$\delta \ne 1$. Moreover, taken verbatim, the first Whitehead Lemma is not true
for arbitrary $\delta$-derivations, as the exceptional cases related to 
$\mathfrak{sl}(2)$, and to the value $\delta = \frac 12$ show. Therefore we 
employ a different approach, which, however, amounts to mere straightforward 
manipulations with the $\delta$-derivation equation (\ref{eq-eqdelta}), and 
utilizing standard facts about semisimple Lie algebras and their representations
(as exposed, for example, in the classical treatises \cite{bourbaki} and 
\cite{jacobson}).

On the other hand, this theorem is a generalization of the result saying that 
all nontrivial $\delta$-derivations of simple finite-dimensional Lie algebras of
characteristic zero are either ordinary derivations ($\delta = 1$), or multiple
of the identity map ($\delta = \frac 12$), or some special family of 
$(-1)$-derivations in the case of $\mathfrak{sl}(2)$ (see 
\cite[Corollary 3]{filippov-98}, \cite[Theorem 2 and Corollary 1]{filippov}, 
or \cite[Corollary 4.6]{leger-luks}).

Our initial interest in such sort of results stems from \cite{octonion}, where 
we computed $\delta$-derivations of certain nonassociative algebras which are of
interest in physics (what, in its turn, helped to determine symmetric 
associative forms on these algebras). These algebras have some classical Lie 
algebras like $\mathfrak{sl}(n)$ and $\mathfrak{so}(n)$ as subalgebras, and 
considering restriction of $\delta$-derivations to these subalgebras, and 
employing the theorem above, would allow to streamline some of the proofs in 
\cite{octonion}.

\section{Auxiliary lemmas}

The proof of the main theorem consists of a series of simple lemmas. The ground 
field $K$ is assumed to be arbitrary, unless stated otherwise.

\begin{lemma}\label{lemma-dirsum}
Let $L$ be a Lie algebra, and let an $L$-module $V$ be decomposable into a 
direct sum of submodules: $V = \bigoplus_i V_i$. Then for any $\delta \in K$,
\begin{equation*}
\Der_\delta(L,V) \simeq \bigoplus_i \Der_\delta(L,V_i) .
\end{equation*}
\end{lemma}

(Here and below the direct sum $\oplus$ is understood in an appropriate 
category: either vector spaces, or Lie algebra modules, or Lie algebras, what 
should be clear from the context).

\begin{proof}
The proof is trivial, and repeats the proof of the similar statement for 
ordinary derivations. Each direct summand $\Der_\delta(L,V_i)$ is obtained by 
composition of $\delta$-derivations from $\Der_\delta(L,V)$ with the canonical 
projection $V \to V_i$.
\end{proof}

\begin{lemma}\label{lemma-sumalg}
Let $L_1$ and $L_2$ be Lie algebras, and $V$ is simultaneously an $L_1$- and 
an $L_2$-module. Then for any nonzero $\delta \in K$,
\begin{multline*}
\Der_\delta(L_1 \oplus L_2, V) \\ \simeq 
\set{(D_1,D_2) \in \Der_\delta(L_1,V) \oplus \Der_\delta(L_2,V)}
{x_1 \bullet D_2(x_2) = x_2 \bullet D_1(x_1) 
\text{ for any } x_1 \in L_1, x_2 \in L_2} .
\end{multline*}
\end{lemma}

(Here and below, $\bullet$ denotes the action of a Lie algebra on its module).

\begin{proof}
Let $D$ be a $\delta$-derivation of $L_1 \oplus L_2$ with values in $V$. Its
restrictions $D_1$ and $D_2$, on $L_1$ and $L_2$ respectively, are, obviously,
$\delta$-derivations with values in $V$, and the condition (\ref{eq-eqdelta})
written for arbitrary pair $x_1 \in L_1$ and $x_2 \in L_2$, is equivalent to the
equality 
\begin{equation}\label{eq-12}
x_1 \bullet D_2(x_2) = x_2 \bullet D_1(x_1) .
\end{equation}
\end{proof}

Recall that a Lie algebra $L$ is called \emph{perfect} if it coincides with its
own commutant: $[L,L] = L$.

\begin{lemma}\label{lemma-n}
Let $L_i$ be a perfect Lie algebra, and $V_i$ an $L_i$-module, $i=1,\dots,n$. 
Then for any $\delta \ne 1$, 
\begin{equation}\label{eq-n}
\Der_\delta(L_1 \oplus \dots \oplus L_n, V_1 \otimes \dots \otimes V_n) \simeq 
\bigoplus_{i=1}^n 
\Big(
V_1^{L_1} \otimes \dots \otimes V_{i-1}^{L_{i-1}} \otimes \Der_\delta(L_i,V_i) 
\otimes V_{i+1}^{L_{i+1}} \otimes \dots \otimes V_n^{L_n}
\Big) .
\end{equation}
\end{lemma}

Here the action of $L_1 \oplus \dots \oplus L_n$ on the tensor product 
$V_1 \otimes \dots \otimes V_n$ is assembled, as usual, from the actions of 
$L_i$'s on $V_1 \otimes \dots \otimes V_n$, where $L_i$ acts on the tensor
factor $V_i$, leaving all other factors intact; $V^L$ denotes the submodule of 
invariants of an $L$-module $V$. 

\begin{proof}
As the claim is trivial in the case $\delta = 0$, we may assume 
$\delta \ne 0,1$. We will prove the case $n=2$, the general case easily follows
by induction.

Since $V_1 \otimes V_2$, as an $L_1$-module, is isomorphic to a direct sum of a
number of copies of $V_1$, parametrized by $V_2$, by Lemma \ref{lemma-dirsum}
we have 
$$
\Der_\delta (L_1, V_1 \otimes V_2) \simeq \Der_\delta(L_1,V_1) \otimes V_2 ,
$$
and similarly, 
$$
\Der_\delta (L_2, V_1 \otimes V_2) \simeq V_1 \otimes \Der_\delta(L_2,V_2) .
$$

Apply Lemma \ref{lemma-sumalg}, assuming $V = V_1 \otimes V_2$. The condition 
(\ref{eq-12}), written for 
$$
\sum_{i\in \mathbb I} D_1^i \otimes v_2^i \in \Der_\delta(L_1,V_1) \otimes V_2 ,
$$
and
$$
\sum_{i\in \mathbb J} v_1^i \otimes D_2^i \in V_1 \otimes \Der_\delta(L_2,V_2) ,
$$
where $D_1^i \in \Der_\delta(L_1,V_1)$, $D_2^i \in \Der_\delta(L_2,V_2)$, $v_1^i \in V_1$, $v_2^i \in V_2$, is 
equivalent to
$$
\sum_{i\in \mathbb J} (x_1 \bullet v_1^i) \otimes D_2^i(x_2) = 
\sum_{i\in \mathbb I} D_1^i(x_1) \otimes (x_2 \bullet v_2^i) .
$$

It follows that one of the following holds: 

(i) $D_1^i(x_1) = x_1 \bullet u_i^1$ for some $u_i^1 \in V_1$, any $x_1 \in L_1$
and any $i\in \mathbb I$, and 
    $D_2^i(x_2) = x_2 \bullet u_i^2$ for some $u_i^2 \in V_2$, any $x_2 \in L_2$
and any $i\in \mathbb J$. 

(ii) $v_1^i \in V_1^{L_1}$ for any $i\in \mathbb I$, and 
     $v_2^i \in V_2^{L_2}$ for any $i\in \mathbb J$.

In the case (i), we have that $D_1^i$ is simultaneously a $\delta$-derivation 
and an $1$-derivation, what implies $(\delta - 1) D_1^i([L,L]) = 0$, and hence 
$D_1^i$ vanishes for any $i\in \mathbb I$. Similarly, $D_2^i$ vanishes for any 
$i\in \mathbb J$. Consequently, we are in case (ii), and
$$
\Der_\delta(L_1 \oplus L_2, V_1 \otimes V_2) \simeq
\Big(\Der_\delta(L_1,V_1) \otimes V_2^{L_2}\Big) \oplus 
\Big(V_1^{L_1} \otimes \Der_\delta(L_2,V_2)\Big) ,
$$
what is the particular case of formula (\ref{eq-n}) for $n=2$.
\end{proof}

\begin{lemma}\label{lemma-4}
In the conditions of Lemma \ref{lemma-n} assume additionally that each $V_i$ is
irreducible. Then \linebreak
$\Der_\delta(L_1 \oplus \dots \oplus L_n, V_1 \otimes \dots \otimes V_n)$ is:
\begin{enumerate}[\upshape(i)]
\item
zero, if either $V_i = K$ for all $i$, or there are at least two different $i$'s
such that $V_i \ne K$;
\item
isomorphic to $\Der(L_i,V_i)$, if there is exactly one $i$ such that 
$V_i \ne K$.
\end{enumerate}
\end{lemma}

\begin{proof}
Since each $V_i$ is irreducible, either $V_i^{L_i} = 0$, or $V_i = K$. Since
$\Der_\delta(L_i,K) = 0$, the claim follows from formula (\ref{eq-n}).
\end{proof}

\begin{lemma}\label{lemma-grad}
Let $L = \bigoplus_{\alpha \in G} L_\alpha$ be a Lie algebra graded by an 
abelian group $G$, and $V = \bigoplus_{\alpha \in G} V_\alpha$ a graded
$L$-module (i.e., $L_\alpha \bullet V_\beta \subseteq V_{\alpha + \beta}$ for 
any $\alpha, \beta \in G$). Then for any $\delta \in K$,
$$
\Der_\delta(L,V) = \bigoplus_{\alpha \in G} \Der_\delta^\alpha (L,V) ,
$$
where 
$$
\Der_\delta^\alpha (L,V) = \set{D \in \Der_\delta(L,V)}{D(L_\beta) \subseteq 
V_{\beta - \alpha} \text{ for any } \beta \in G} .
$$
\end{lemma}

\begin{proof}
Exactly the same simple arguments as in the case of ordinary derivations (see, 
for example, \cite[Proposition 1.1]{farnsteiner}).
\end{proof}

$\delta$-derivations from the space $\Der_\delta^\alpha(L,V)$ are said to be
\emph{of weight $\alpha$}.

\begin{lemma}\label{lemma-cent}
Let $L$ be a simple Lie algebra, $V$ an irreducible $L$-module, and $D: L \to V$
a nonzero linear map such that $D([x,y]) = x \bullet D(y)$ for any $x, y \in L$.
Then $V$ is isomorphic to the adjoint module.
\end{lemma}

\begin{proof}
By definition, $D$ is a homomorphism of $L$-modules from the adjoint module $L$ 
to $V$. Obviously, $\Ker D$ is an ideal in $L$, thus $D$ is an injection. On the
other hand, $\im D$ is an $L$-submodule of $V$, thus $D$ is a surjection.
\end{proof}

\begin{lemma}\label{lemma-abelian}
Any $\delta$-derivation, where $\delta \ne 0$, of an abelian Lie algebra $H$
with values in a semisimple $H$-module $V$, is of the form
$x \mapsto \varphi(x) + x \bullet v$, where $x\in H$, for some linear map
$\varphi: H \to V^H$, and some $v\in V$.
\end{lemma}

\begin{proof}
By extending the ground field, we may assume that $K$ is algebraically closed. 
Let $D$ be such a $\delta$-derivation. Suppose first that $V$ is a
one-dimensional nontrivial $H$-module, linearly spanned by a single element $v$.
We may write $x \bullet v = \lambda(x)v$ and $D(x) = \mu(x) v$ for any $x\in H$
and some linear maps $\lambda, \mu: H \to K$. Then the condition 
(\ref{eq-eqdelta}) is equivalent to
\begin{equation}\label{eq-lm}
\lambda(x) \mu(y) = \lambda(y) \mu(x)
\end{equation}
for any $x,y\in H$. If $\Ker \lambda \ne \Ker \mu$, then taking $x \in H$ 
belonging to $\Ker \lambda$ and not belonging to $\Ker \mu$, we get from 
(\ref{eq-lm}) that $y \in \Ker\lambda$ for any $y \in H$, thus $V$ is the 
trivial module, a contradiction. Hence $\Ker \lambda = \Ker \mu$, and $\lambda$
and $\mu$ are proportional to each other, say, $\mu = \alpha \lambda$ for some 
$\alpha \in K$. Then $D(x) = x \bullet (\alpha v)$, and the assertion of lemma 
in this case follows.

In the general case $V$ can be represented as the direct sum of the trivial 
module $V^H$ and a number of one-dimensional nontrivial $H$-modules: say, 
$V = V^H \oplus \bigoplus_i Kv_i$. By Lemma \ref{lemma-dirsum},
$$
\Der_\delta(H,V) \simeq 
\Der_\delta(H,V^H) \oplus \bigoplus_i \Der_\delta(H,Kv_i) .
$$

The space $\Der_\delta(H,V^H)$, obviously, coincides with the space of all linear maps
$\varphi: H \to V^H$, and by the just proved one-dimensional case, we may assume
that any element of $\Der_\delta(H,Kv_i)$ is of the form 
$x \mapsto x \bullet v_i$. Then for any $x \in H$,
$$
D(x) = \varphi(x) + \sum_i (x \bullet v_i) = 
\varphi(x) + x \bullet (\sum_i v_i) ,
$$
and we are done.
\end{proof}

In what follows, for a finite-dimensional simple Lie algebra $\mathfrak g$
over an algebraically closed field $K$ of characteristic zero, we fix once and
for all a Cartan subalgebra $\mathfrak h$, and the corresponding root space 
decomposition 
$\mathfrak g = \mathfrak h \oplus \bigoplus_{\alpha \in R} Ke_\alpha$.

\begin{lemma}\label{lemma-u}
Let $D$ be a nonzero $\delta$-derivation of a finite-dimensional simple Lie
algebra $\mathfrak g$ over an algebraically closed field of characteristic zero,
with values in a finite-dimensional irreducible $\mathfrak g$-module $V$. Assume
that there is nonzero $v \in V$, and $\beta \in R$ such that 
$e_\alpha \bullet v = 0$ and $D(e_\alpha) = 0$ for any $\alpha \in R$ such that $\alpha \ne \beta$. Then
$\mathfrak g \simeq \mathfrak{sl}(2)$ and $\dim V \le 3$.
\end{lemma}

\begin{proof}
Decomposing, if necessary, $v$ into the sum of nonzero elements belonging to 
weight spaces, we may assume that $v$ belongs to some weight space $V_\lambda$.
Replacing, if necessary, $\beta$ by $-\beta$ and $\lambda$ by $-\lambda$, we get
that $\lambda$ is the highest weight. Since $V$ is irreducible, it is generated,
as a module over the universal enveloping algebra $U(\mathfrak g)$, by a single
element $v$. By the Poincar\'e--Birkhoff--Witt theorem, each element of 
$U(\mathfrak g)$ is a sum of elements of the form 
\begin{equation*}
e_\beta^k h_1^{k_1} \dots h_n^{k_n} 
e_{\alpha_1}^{\ell_1} \dots e_{\alpha_m}^{\ell_m}
\end{equation*}
for some $h_1, \dots, h_n \in \mathfrak h$, $\alpha_1, \dots, \alpha_m \in R$, 
each $\alpha_i \ne \beta$, and some nonnegative integers 
$k,$ $k_1, \dots, k_n,$ $\ell_1, \dots, \ell_m$. If such an element really 
contains $e_\alpha$'s (that is, at least one of $\ell$'s is positive), then it 
acts on $v$ trivially, and since $h$'s act on $v$ by multiplying it by a scalar,
we get that $U(\mathfrak g)v$ is linearly spanned by elements of the form 
$e_\beta^k \bullet v$, $k=0,1,2,\dots$. But since $\lambda$ is the highest 
weight, $e_\beta^k \bullet v = 0$ if $k \ge 3$, thus $V$ is at most 
$3$-dimensional, linearly spanned by elements 
$v, e_\beta \bullet v, e_\beta^2 \bullet v$.

(At this point, the dimension considerations imply that 
$\dim \mathfrak g \le 9$ and hence $\mathfrak g$ is isomorphic either to 
$\mathfrak{sl}(2)$, or to $\mathfrak{sl}(3)$, but this obvious remark is 
superseded by the reasoning below).

Since $\lambda$ is the highest weight, we may assume $\beta = -\rho$, where 
$\rho$ is the highest root, and if rank of $\mathfrak g$ is $>1$, then 
$\beta = \alpha + \alpha^\prime$ for some $\alpha, \alpha^\prime \in R$. Then
the equation (\ref{eq-eqdelta}), written for $x = e_\alpha$, 
$y = e_{\alpha^\prime}$, yields $D(e_\beta) = 0$, a contradiction. Hence 
$\mathfrak g$ is of rank $1$, i.e. $\mathfrak g \simeq \mathfrak{sl}(2)$.
\end{proof}

\section{The case of $\mathfrak{sl}(2)$}

In this section we shall prove the main theorem in the case of 
$\mathfrak{sl}(2)$. Let the characteristic of the ground field be zero, 
$\{e_-, h, e_+\}$ be the standard basis of $\mathfrak{sl}(2)$ with 
multiplication table 
$$
[h,e_-] = -2e_-, \quad [h,e_+] = 2e_+, \quad [e_+,e_-] = h .
$$
The algebra $\mathfrak{sl}(2)$ is $\mathbb Z$-graded. We assign to elements of 
the standard basis the weights $1$, $0$, $-1$, respectively.

Let $V(n)$ denote the irreducible $(n+1)$-dimensional (i.e., of the highest 
weight $n$) $\mathfrak{sl}(2)$-module with the standard basis 
$\{v_0, v_1, \dots, v_n \}$. The action is given as follows:
\begin{align*}
e_- \bullet v_i &= (i+1) \> v_{i+1}             \\
h\phantom{_+} \bullet v_i  &= (n - 2i) \> v_i   \\
e_+ \bullet v_i &= (n - i + 1) \> v_{i-1}
\end{align*}
(Customarily, here and below we assume $v_i = 0$ if $i$ is out of range 
$0,\dots,n$). This is a graded module, with element $v_i$ having weight $i$. 
(Note that our assignment of weights in $\mathfrak{sl}(2)$ and $V(n)$ is not a 
standard one, but is slightly more convenient for keeping track of indices in 
computations below). Note also that $V(2) \simeq \mathfrak{sl}(2)$, the adjoint
module.

\begin{lemma}\label{lemma-2}
$\Der_\delta(\mathfrak{sl}(2), V(n)) = 0$ except for the following cases (in all
these cases, $n \ge 1$):
\begin{enumerate}[\upshape(i)]
\item\label{it-1}
$\delta = 1$; the space $\Der_1(\mathfrak{sl}(2),V(n))$ is $(n+1)$-dimensional, 
consisting of inner derivations.

\medskip

\item\label{it-2n}
$\delta = -\frac 2n$; the space $\Der_{-\frac 2n}(\mathfrak{sl}(2),V(n))$ is
$(n+3)$-dimensional, with a basis
{\rm \begin{center}
\begin{tabular}{p{70pt}p{70pt}p{70pt}p{70pt}p{70pt}}
$e_- \mapsto 0$        &
$e_- \mapsto 0$        &
$e_- \mapsto -v_{k+1}$ &
$e_- \mapsto -v_1$     &
$e_- \mapsto v_0$
\\
$h\phantom{_+} \mapsto 0$     &
$h\phantom{_+} \mapsto 2v_n$  &
$h\phantom{_+} \mapsto 2v_k$  &
$h\phantom{_+} \mapsto 2v_0$  &   
$h\phantom{_+} \mapsto 0$
\\
$e_+ \mapsto v_n$      &
$e_+ \mapsto v_{n-1}$  &
$e_+ \mapsto v_{k-1}$  &
$e_+ \mapsto 0$        &
$e_+ \mapsto 0$
\end{tabular}
\end{center}}
consisting of $(-\frac 2n)$-derivations of weight 
$-n-1$, $-n$, $-k$, $0$, $1$, respectively, where $1 \le k \le n-1$.

\medskip

\item\label{it-2n2}
$\delta = \frac{2}{n+2}$ and $n \ge 2$; the space 
$\Der_{\frac{2}{n+2}}(\mathfrak{sl}(2),V(n))$ is $(n-1)$-dimensional, with a 
basis
{\rm \begin{center}
\begin{tabular}{p{150pt}}
$e_- \mapsto k(k+1) \, v_{k+1}$
\\
$h\phantom{_+} \mapsto 2k(n-k) \, v_k$
\\
$e_+ \mapsto -(n-k)(n-k+1) \, v_{k-1}$
\end{tabular}
\end{center}}
consisting of $\frac{2}{n+2}$-derivations of weight $-k$, where 
$1 \le k \le n-1$.

\end{enumerate}
\end{lemma}

As $V(2) \simeq \mathfrak{sl}(2)$, for $n=2$ we are dealing with just 
$\delta$-derivations of $\mathfrak{sl}(2)$, and this particular case of 
Lemma~\ref{lemma-2} was known. Indeed, it follows from the results of 
\cite{filippov-98} and \cite{filippov} mentioned in the introduction, that 
$\mathfrak{sl}(2)$ has nonzero $\delta$-derivations only in the cases 
$\delta = 1$ (the ordinary inner derivations), $\delta = \frac 12$ (scalar 
multiples of the identity map, a particular case of (\ref{it-2n2})), and 
$\delta = -1$; and the $5$-dimensional space 
$\Der_{-1}(\mathfrak{sl}(2), \mathfrak{sl}(2))$, a particular case of 
(\ref{it-2n}), was described in \cite[Example 1.5]{hopkins} and 
\cite[Example in \S 3]{filippov-5}.

\begin{proof}
As we are dealing with a $\mathbb Z$-graded module over a $\mathbb Z$-graded Lie
algebra, by Lemma \ref{lemma-grad} it is sufficient to consider 
$\delta$-derivations of a fixed weight $\alpha \in \mathbb Z$. So, let $D$ be
a nonzero map lying in $\Der_\delta^\alpha (\mathfrak{sl}(2), V(n))$. Note that 
$\delta \ne 0$ and $n>0$. Also, the case $\delta = 1$ corresponds to the usual
derivations (= $1$-cocycles), case (\ref{it-1}), so when encountered in the 
computations below, it can be readily discarded.

We have
\begin{alignat*}{3}
&D(e_-) &\>=\>& \lambda v_{1-\alpha} \\
&D(h)   &\>=\>& \mu v_{-\alpha}         \\
&D(e_+) &\>=\>& \eta v_{-1 - \alpha}
\end{alignat*}
for some $\lambda, \mu, \eta \in K$. To ensure that at least one of the indices
is in the range $0,\dots,n$, we have $-n-1 \le \alpha \le 1$. Writing the 
equation (\ref{eq-eqdelta}) for all possible $3$ pairs of the basis elements 
$(h,e_-)$, $(e_+,e_-)$, and $(h,e_+)$, we get respectively:
\begin{gather}
\Big(-2\lambda + \delta \mu (1 - \alpha) - \delta\lambda (n - 2 + 2\alpha)\Big) 
v_{1-\alpha} = 0 \notag
\\
\Big(\mu - \delta\eta\alpha - \delta\lambda (n + \alpha)\Big) v_{-\alpha} = 0
\label{eq-g}
\\
\Big(2\eta + \delta\mu (n+\alpha+1) - \delta\eta (n+2+2\alpha)\Big) 
v_{-1-\alpha} = 0 \notag
\end{gather}

\emph{Case 1}. $\alpha = -n-1$. The first and the second equations in 
(\ref{eq-g}) give nothing, and the third one is equivalent to 
$\eta(2 + n\delta) = 0$. Since $D$ is not zero, we may normalize it by assuming
$\eta = 1$, what gives $\delta = -\frac{2}{n}$. This is the first 
$(-\frac 2n)$-derivation in case (\ref{it-2n}).

\emph{Case 2}. $\alpha = -n$. The first equation in (\ref{eq-g}) gives nothing,
and the second and the third one are equivalent to:
\begin{equation}\label{eq-q1}
\mu + n\delta\eta = 0
\end{equation}
and
\begin{equation}\label{eq-q2}
2\eta + \delta\mu + (n-2)\delta\eta = 0 ,
\end{equation}
respectively. If $\eta = 0$, then $\mu = 0$, a contradiction. Hence we may 
normalize $D$ by assuming $\eta = 1$, and resolving the quadratic equation in 
$\delta$ occurring from equations (\ref{eq-q1})--(\ref{eq-q2}), we get that 
either $\delta = 1$, or $\delta = -\frac 2n$ and $\mu = 2$, and we are getting 
the second $(-\frac 2n)$-derivation in case (\ref{it-2n}).

\emph{Case 3}. $-n+1 \le \alpha \le -1$. Note that this implies $n\ge 2$. All 
the indices of $v$'s occurring in (\ref{eq-g}) are within the allowed range, 
thus all the coefficients of $v$'s vanish. Consider these vanishing conditions 
as a system of $3$ homogeneous linear equations in $\lambda$, $\mu$, $\eta$. The determinant of this system is a 
cubic equation in $\delta$, whose roots are not dependent on $\alpha$, and are
equal to $-\frac 2n$, $\frac{2}{n+2}$, and $1$.

\emph{Case 3.1}. $\delta = -\frac 2n$. The space of solutions of the homogeneous
system is $1$-dimensional, linearly spanned by 
$(\lambda = -1, \mu=2, \eta = 1)$. This is the family of 
$(-\frac 2n)$-derivations in case (\ref{it-2n}), with $k=-\alpha$.

\emph{Case 3.2}. $\delta = \frac{2}{n+2}$. The space of solutions of the
homogeneous system is $1$-dimensional, linearly spanned by
$$
\lambda = \alpha(\alpha - 1),\>  \mu = -2\alpha(n+\alpha),\> 
\eta = -(n+\alpha)(n+\alpha+1) .
$$
These are exactly $\frac{2}{n+2}$-derivations in case (\ref{it-2n2}), with 
$k = -\alpha$.

\emph{Case 4}. $\alpha = 0$. The third equation in (\ref{eq-g}) gives nothing,
and the first and the second are equivalent to
\begin{equation}\label{eq-yoyo}
-2\lambda + \delta \mu - (n-2)\delta\lambda = 0
\end{equation}
and
\begin{equation}\label{eq-yoyo1}
\mu - n\delta\lambda = 0 ,
\end{equation}
respectively. If $\lambda = 0$, then (\ref{eq-yoyo1}) implies $\mu = 0$, a
contradiction. Hence we may normalize $D$ by assuming $\lambda = -1$, and then
the quadratic equation in $\delta$ occurring from 
(\ref{eq-yoyo})--(\ref{eq-yoyo1}) gives that either $\delta = 1$, or 
$\delta = - \frac 2n$ and $\mu = 2$. In the last case we get the fourth
$(-\frac 2n)$-derivation in case (\ref{it-2n}).

\emph{Case 5}. $\alpha = 1$. The second and the third equations in (\ref{eq-g})
give nothing, and the first equation is equivalent to 
$\lambda (2 + n\delta) = 0$. Since $D$ is nonzero, we may normalize it by 
assuming $\lambda = 1$, and then $\delta = -\frac 2n$, what gives the last, 
fifth, $(-\frac 2n)$-derivation in case (\ref{it-2n}).
\end{proof}

\section{The case of $\mathfrak g \not\simeq \mathfrak{sl}(2)$}

In the previous section we proved the main theorem in the case of 
$\mathfrak{sl}(2)$ and an irreducible $\mathfrak{sl}(2)$-module. Now we are 
ready to handle the case of any other simple Lie algebra $\mathfrak g$ and an irreducible 
$\mathfrak g$-module.

\begin{lemma}\label{lemma-simple}
Let $D$ be a nonzero $\delta$-derivation, $\delta \ne 1$, of a simple 
finite-dimensional Lie algebra $\mathfrak g$ over an algebraically closed field
of characteristic zero, $\mathfrak g \not\simeq \mathfrak{sl}(2)$, with values 
in a finite-dimensional irreducible $\mathfrak g$-module $V$. Then 
$\delta = \frac 12$, $V$ is isomorphic to the adjoint module, and $D$ is a 
multiple of the identity map.
\end{lemma}

\begin{proof}
If $V$ is trivial, then the condition (\ref{eq-eqdelta}) implies $D([x,y]) = 0$
for any $x,y\in \mathfrak g$, thus $D=0$, a contradiction; so we may assume that
$V$ is nontrivial, i.e. $\dim V \ge 2$.

The Cartan subalgebra $\mathfrak h$ acts on $V$ semisimply, with the 
corresponding weight space decomposition 
$V = \bigoplus_{\beta \in \Phi} V_\beta$. By Lemma \ref{lemma-abelian}, we can 
write $D(h) = \varphi(h) + h \bullet v$ for any $h\in \mathfrak h$, some linear
map $\varphi: \mathfrak h \to V^{\mathfrak h} = V_0$, and some $v \in V$. 
Writing the equality (\ref{eq-eqdelta}) for $x = h \in \mathfrak h$ and 
$y = e_\alpha$, $\alpha \in R$, we get
\begin{equation}\label{eq-he}
\alpha(h)D(e_\alpha) = -\delta e_\alpha \bullet (\varphi(h) + h \bullet v) 
+ \delta h \bullet D(e_\alpha) .
\end{equation}

By Lemma \ref{lemma-grad}, it is enough to consider $\delta$-derivations of some
weight $\gamma \in \langle R \rangle$, where $\langle R \rangle$ is the abelian
group generated by $R$ (note that $\Phi \subset \langle R \rangle$). We consider
several cases depending on the weight of $D$.

\emph{Case 1}. $D$ is of zero weight. In this case 
$D(\mathfrak h) \subseteq V_0$ and $D(e_\alpha) \in V_\alpha$ for any 
$\alpha \in R$ (here and below we assume $V_\lambda = 0$ if 
$\lambda \notin \Phi$). Then $v \in V_0$, thus $h \bullet v = 0$, $D = \varphi$,
$h \bullet D(e_\alpha) = \alpha(h) D(e_\alpha)$, and the equality (\ref{eq-he})
reduces to
\begin{equation}\label{eq-he1}
\alpha(h)D(e_\alpha) = \frac{\delta}{\delta - 1} e_\alpha \bullet D(h) .
\end{equation}

Writing the equality (\ref{eq-eqdelta}) for $x = e_\alpha$ and $y = e_\beta$,
where $\alpha, \beta \in R$, $\alpha + \beta \ne 0$, and taking into account
(\ref{eq-he1}) and the fact that $[e_\alpha,e_\beta] \in Ke_{\alpha + \beta}$, 
we get
$$
- \big((\delta - 1)\beta(h)  + \delta\alpha(h)\big) 
e_\alpha \bullet D(e_\beta)
+
  \big((\delta - 1)\alpha(h) + \delta\beta(h)\big) 
e_\beta \bullet D(e_\alpha)
= 0
$$
for any $h \in \mathfrak h$. Assuming the rank of $\mathfrak g$ is $>1$, 
$\alpha \ne \beta$, and picking $h$ such that $\alpha(h) \ne 0$ and 
$\beta(h) = 0$, the last equality is reduced to
\begin{equation}\label{eq-ab}
\delta e_\alpha \bullet D(e_\beta) + (1 - \delta) e_\beta \bullet D(e_\alpha)
= 0 .
\end{equation}

Assuming $\delta \ne \frac 12$, and interchanging here $\alpha$ and $\beta$, we
get $e_\beta \bullet D(e_\alpha) = 0$ for any $\alpha, \beta \in R$ such that
$\alpha + \beta \ne 0$. Now reasoning as in the first half of the proof of 
Lemma \ref{lemma-u}, we get that $\dim V = 3$, and $V$ is spanned by elements
of weight $\alpha, 0, -\alpha$. Since the rank of $\mathfrak g$ is $>1$, we can
pick two linearly independent roots $\alpha$ and $\beta$, thus getting different
sets of weights of elements of $V$, a contradiction.

We are left with the case $\delta = \frac 12$. In this case, (\ref{eq-he1}) can
be rewritten as
\begin{equation}\label{eq-he2}
h \bullet D(e_\alpha) + e_\alpha \bullet D(h) = 0
\end{equation}
for any $h \in \mathfrak h$ and $\alpha \in R$, and (\ref{eq-ab}) is equivalent
to
\begin{equation}\label{eq-ab2}
e_\alpha \bullet D(e_\beta) + e_\beta \bullet D(e_\alpha) = 0
\end{equation}
for any $\alpha, \beta \in R$, $\alpha + \beta \ne 0$. Moreover, since 
$[e_\alpha, e_{-\alpha}] \in \mathfrak h$ acts on $D(h)$ trivially for any 
$h \in \mathfrak h$, we have 
$$
  e_\alpha    \bullet (e_{-\alpha} \bullet D(h)) 
- e_{-\alpha} \bullet (e_\alpha    \bullet D(h)) = 0 ,
$$
what, taking into account (\ref{eq-he1}) and assuming $\alpha(h) \ne 0$, yields
\begin{equation}\label{eq-am2}
e_\alpha \bullet D(e_{-\alpha}) + e_{-\alpha} \bullet D(e_{\alpha}) = 0
\end{equation}
for any $\alpha \in R$.

The equalities (\ref{eq-he2}), (\ref{eq-ab2}), and (\ref{eq-am2}) show that for
any $x,y \in \mathfrak g$, $x \bullet D(y) + y \bullet D(x) = 0$, and hence
$D([x,y]) = x \bullet D(y)$. By Lemma \ref{lemma-cent}, $V$ is isomorphic to the
adjoint module, and thus $D$ is an element of the centroid of $\mathfrak g$. But
since $\mathfrak g$ is simple (and hence central), $D$ is a multiple of the identity map.

\emph{Case 2}. $D$ is of nonzero weight $\gamma \notin R$. In this case 
$D(\mathfrak h) = 0$ and $D(e_\alpha) \in V_{\alpha - \gamma}$ for any 
$\alpha \in R$, and (\ref{eq-he}) reduces to
$$
\alpha(h)D(e_\alpha) = \delta h \bullet D(e_\alpha) .
$$

This means that for any $\alpha\in R$, either $D(e_\alpha) = 0$, or 
$D(e_\alpha) \in V_{\frac{1}{\delta}\alpha}$. The latter condition implies 
$\alpha - \gamma = \frac{1}{\delta} \alpha$, and
$\alpha = \frac{\delta}{\delta - 1} \gamma$. Writing the equality 
(\ref{eq-eqdelta}) for $x = e_\alpha$, 
$\alpha \ne \frac{\delta}{\delta - 1} \gamma$, and 
$y = e_{\frac{\delta}{\delta - 1} \gamma}$, we get
$$
e_\alpha \bullet D(e_{\frac{\delta}{\delta - 1} \gamma}) = 0 .
$$

Note that $D(e_{\frac{\delta}{\delta - 1} \gamma}) \ne 0$, otherwise $D$ is the
zero map. Therefore the conditions of Lemma \ref{lemma-u} are satisfied, with
$v = D(e_{\frac{\delta}{\delta - 1} \gamma})$ and 
$\beta = \frac{\delta}{\delta - 1} \gamma$, and hence 
$\mathfrak g \simeq \mathfrak{sl}(2)$, a contradiction.

\emph{Case 3}. $D$ is of nonzero weight $\beta \in R$. In this case 
$D(\mathfrak h) \subseteq V_{-\beta}$ and $D(e_\alpha) \in V_{\alpha - \beta}$
for any $\alpha \in R$. Then $\varphi = 0$ and $v \in V_{-\beta}$, thus 
$D(h) = h \bullet v = -\beta(h)v$, and the equality
(\ref{eq-he}) reduces to
\begin{equation}\label{eq-alpha}
\alpha(h)D(e_\alpha) = 
\delta \beta(h) e_\alpha \bullet v + \delta h \bullet D(e_\alpha)
\end{equation}
for any $\alpha \in R$. Assume here $\alpha - \beta \in \Phi$ and 
$\alpha \ne \beta$. Since $D(e_\alpha) \in V_{\alpha - \beta}$, we have 
$$
h \bullet D(e_\alpha) = (\alpha(h) - \beta(h))D(e_\alpha) ,
$$
and the equality (\ref{eq-alpha}) in this case is equivalent to
\begin{equation*}
\Big((1 - \delta)\alpha(h) + \delta \beta(h)\Big) D(e_\alpha) = 
\delta \beta(h) e_\alpha \bullet v .
\end{equation*}

Picking $h$ such that $\beta(h) = 0$ and $\alpha(h) \ne 0$, we get 
$D(e_\alpha) = 0$. Since the latter equality is true also for any $\alpha \in R$
such that $\alpha - \beta \notin \Phi$, we have $D(e_\alpha) = 0$ for any 
$\alpha \in R$, $\alpha \ne \beta$. It follows then from (\ref{eq-alpha}) that 
$e_\alpha \bullet v = 0$ for any $\alpha \ne \beta$. Therefore, again, we are
in the conditions of Lemma \ref{lemma-u}, thus
$\mathfrak g \simeq \mathfrak{sl}(2)$, a contradiction.
\end{proof}

\begin{remark}
A slight shortcut in part of the proof of Lemma \ref{lemma-simple} can be 
achieved by invoking \cite[Lemma 4.4]{delta} which implies that if a simple Lie
algebra $L$ has a nontrivial $\delta$-derivation with values in an $L$-module, 
then either $L$ satisfies the standard identity of degree $5$, or $\delta = 1$ 
or $\frac 12$. But \cite[Lemma 4.4]{delta} is just a slight upgrade of 
\cite[Theorem 1]{filippov} which involves sophisticated manipulations with 
identities, and, in its turn, is based on highly nontrivial results of Razmyslov
about identities in simple Lie algebras. And still, we will have to handle the 
$\delta = \frac 12$ part in Case 1 of the proof, and Cases 2 and 3 entirely. 
Thus we prefer a more direct approach, based on relatively trivial computations
with root systems in simple Lie algebras.

Another possibility would be to consider subalgebras of $\mathfrak g$ isomorphic
to $\mathfrak{sl}(2)$, and invoke Lemma \ref{lemma-2}, what would restrict our
considerations to the exceptional cases of $\delta$ described there, which still
had to be handled separately, more or less along the lines of the given proof.
\end{remark}

\section{Completion of the proof of the main theorem}

Lemmas \ref{lemma-2} and \ref{lemma-simple} together establish the main theorem
stated in the introduction, in the case of a simple $\mathfrak g$ and an 
irreducible $\mathfrak g$-module. Now, basing of this, we complete the proof for
the general case of a semisimple $\mathfrak g$ and arbitrary $\mathfrak g$-module.

If $\delta = 1$, the statement reduces to the ordinary derivations, and, as 
noted above, is equivalent to the first Whitehead Lemma about triviality of the
first Chevalley--Eilenberg cohomology $\Homol^1(\mathfrak g, V)$. So we may 
assume $\delta \ne 1$. Also, obviously, $\delta \ne 0$.

Let $\mathfrak g = \mathfrak g_1 \oplus \dots \oplus \mathfrak g_m$ be the 
direct sum decomposition into simple ideals $\mathfrak g_i$. Assume first that 
$V$ is irreducible. Then $V \simeq V_1 \otimes \dots \otimes V_m$, where $V_i$ 
is an irreducible $\mathfrak g_i$-module. By Lemma \ref{lemma-4}, exactly one of
the modules $V_1, \dots, V_m$, say, $V_k$, is different from the trivial module,
and the rest $m-1$ modules are trivial. Therefore $V \simeq V_k$ is an 
irreducible nontrivial $\mathfrak g_k$-module, and $\mathfrak g_i$ acts on $V$ 
trivially if $i \ne k$. The restriction of any $\delta$-derivation 
$D: \mathfrak g \to V$ on $\mathfrak g_k$ is a $\delta$-derivation of 
$\mathfrak g_k$ with values in $V_k$, $D$ acts trivially on the rest of 
$\mathfrak g_i$, $i \ne k$, and 
$\Der_\delta(\mathfrak g, V) \simeq \Der_\delta(\mathfrak g_k, V_k)$. By 
Lemmas \ref{lemma-2} and \ref{lemma-simple}, one of the following holds:
\begin{enumerate}[\upshape(i)]

\item
$\delta = -\frac 2n$, $n \ge 1$, $\mathfrak g_k \simeq \mathfrak{sl}(2)$, 
$V_k \simeq V(n)$, and $\Der_\delta(\mathfrak g_k, V_k)$ is as in 
Lemma~\ref{lemma-2}(\ref{it-2n});

\item\label{it-2n21}
$\delta = \frac{2}{n+2}$, $n\ge 2$, $\mathfrak g_k \simeq \mathfrak{sl}(2)$, 
$V_k \simeq V(n)$, and $\Der_\delta(\mathfrak g_k, V_k)$ is as in 
Lemma~\ref{lemma-2}(\ref{it-2n2});

\item\label{it-id}
$\delta = \frac 12$, $V_k \simeq \mathfrak g_k$ (the adjoint module), and 
$\Der_\delta(\mathfrak g_k, V_k) \simeq K\id_{\mathfrak g_k}$.

\end{enumerate}

Note that the cases (\ref{it-2n21}) and (\ref{it-id}) overlap if $n=2$ and 
$\mathfrak g_k \simeq \mathfrak{sl}(2)$.

The proof is finished by the remark that in the general case $V$ is decomposed
into the direct sum of irreducible $\mathfrak g$-modules, and application of 
Lemma \ref{lemma-dirsum}.

\section{
Open questions: positive characteristic and infinite-dimensional modules
}

What happens in positive characteristic? As it is commonly known, the general
situation in this case is much more complicated. The first Whitehead Lemma does
not hold (in a sense, the opposite is true: as proved in 
\cite[Chapter VI, \S 3, Theorem 2]{jacobson}, any finite-dimensional Lie algebra
over a field of positive characteristic has a finite-dimensional module with 
first nonzero cohomology), so there is no point to conjecture that (most of) 
$\delta$-derivations should be inner derivations. However, this does not 
preclude the possibility that ``most'', in some sense, $\delta$-derivations can
be ordinary derivations.

\begin{conjecture}
A nonzero $\delta$-derivation of a finite-dimensional semisimple Lie algebra $L$
over a field of characteristic $\ne 2,3$ with values in a finite-dimensional 
$L$-module is, excluding the small number of exceptional cases, either an 
ordinary derivation ($\delta = 1$), or is a multiple of the identity map on a 
simple component of $L$ with coefficients in the adjoint module 
($\delta = \frac 12$).
\end{conjecture}

To tackle this conjecture, one should overcome a number of difficulties
specific to positive characteristic. First, in addition to $\mathfrak{sl}(2)$,
there is another class of simple Lie algebras of rank one having nontrivial 
$\delta$-derivations with $\delta \ne 1$ -- namely, Zassenhaus algebras having 
nontrivial $\frac 12$-derivations (see \cite[\S 3]{filippov} and 
\cite[\S 2]{delta}); this suggests that the number of exceptional cases will be
(much) higher than in characteristic zero. Second, the description of 
irreducible representations of simple Lie algebras -- except for the simplest 
cases of algebras of low rank -- is lacking (and, quite possibly, such full 
description is hopeless). Third, arbitrary representations of simple Lie 
algebras are, generally, not direct sums of irreducibles. And fourth, semisimple
Lie algebras are, generally, not direct sums of simples.

The last difficulty, however, could be overcome using the classical Block 
theorem which says that in positive characteristic, semisimple Lie algebras are
sitting between the direct sum of Lie algebras of the form 
$S \otimes \mathcal O$, where $S$ is a simple Lie algebra, and $\mathcal O$ is a
divided powers algebra, and their derivation algebras. This, basically, reduces
questions about semisimple algebras to semidirect sums of the form
\begin{equation}\label{eq-sod}
S \otimes \mathcal O + \mathcal D ,
\end{equation}
where $\mathcal D$ is a derivation algebra of $\mathcal O$. It is proved in 
\cite[\S 1]{delta} that the space of $\delta$-derivations of the Lie algebra of
the form $L \otimes A$, where $L$ is a Lie algebra, and $A$ an associative 
commutative algebra, is, essentially, reduced to the space 
$\Der_\delta(L) \otimes A$, and an extension of this result to 
$\delta$-derivations of Lie algebras of the form (\ref{eq-sod}) with values in more-or-less arbitrary modules would, hopefully, allow to
reduce Conjecture from the semisimple case to the simple one.

Going back to characteristic zero, it would be also interesting to compute 
$\delta$-derivations of a (semi)sim\-ple finite-dimensional Lie algebra
$\mathfrak g$ with values in infinite-dimensional $\mathfrak g$-modules. The 
first cohomology of $\mathfrak g$ in this case does not necessarily vanish: this
follows from abstract nonsense (otherwise $\mathfrak g$ would be of 
cohomological dimension $1$), and for concrete examples of infinite-dimensional
$\mathfrak g$-modules with nonvanishing (first) cohomology see \cite{bavula},
\cite{williams}, and references therein.

\section*{Acknowledgments}

GAP \cite{gap}\footnote{
Using a Perl wrapper for GAP for solving linear equations on (nonassociative) 
algebras, available at \newline 
\texttt{https://web.osu.cz/$\sim$Zusmanovich/soft/lineq/ }
}
and Maxima \cite{maxima} were used to verify some of the computations performed
in this paper.
Arezoo Zohrabi was supported by grant SGS01/P\v{r}F/20-21 of the University of 
Ostrava.

\renewcommand{\refname}{Software}

\end{document}